\documentclass[11pt, a4paper, oneside]{amsart}

\usepackage[english]{babel}
\usepackage{amsmath, amsthm, amsfonts, mathrsfs, amssymb}
\usepackage{mathtools}
\mathtoolsset{centercolon}
\usepackage{booktabs}
\usepackage[shortlabels]{enumitem}
\usepackage[colorlinks, citecolor = blue]{hyperref}
\usepackage{fullpage}
\usepackage{comment}
\usepackage{color}
\usepackage{graphicx}
\usepackage{todonotes}

\usepackage{tikz}
\usetikzlibrary{positioning,calc,arrows.meta}
\definecolor{knownblue}{RGB}{49,104,142}
\definecolor{newred}{RGB}{192,57,43}
\definecolor{boxfill}{RGB}{248,249,250}
\definecolor{boxline}{RGB}{90,90,90}

\newcommand{\N}{\ensuremath{\mathbb{N}}}

\newcommand{\R}{\ensuremath{\mathbb{R}}}
\newcommand{\C}{\ensuremath{\mathbb{C}}}

\newcommand{\E}{\ensuremath{\mathbb{E}}}
\renewcommand{\P}{\ensuremath{\mathbb{P}}}

\newcommand{\mc}{\mathcal}

\DeclarePairedDelimiter\abs{\lvert}{\rvert}

\DeclarePairedDelimiter\cbrace\{\}

\DeclarePairedDelimiter{\nrm}\lVert\rVert

\newcommand{\absb}[1]{\bigl|#1\bigr|}
\newcommand{\brb}[1]{\bigl(#1\bigr)}
\newcommand{\cbraceb}[1]{\bigl\{#1\bigr\}}

\newcommand{\absB}[1]{\Bigl|#1\Bigr|}
\newcommand{\brB}[1]{\Bigl(#1\Bigr)}

\DeclareMathOperator{\loc}{loc}

\DeclareMathOperator{\supp}{supp}
\DeclareMathOperator{\ind}{\mathbf{1}}

\newcommand{\avgL}{\textit{\L}}

\makeatletter
\DeclareFontFamily{OMX}{MnSymbolE}{}
\DeclareSymbolFont{MnLargeSymbols}{OMX}{MnSymbolE}{m}{n}
\SetSymbolFont{MnLargeSymbols}{bold}{OMX}{MnSymbolE}{b}{n}
\DeclareFontShape{OMX}{MnSymbolE}{m}{n}{
    <-6>  MnSymbolE5
   <6-7>  MnSymbolE6
   <7-8>  MnSymbolE7
   <8-9>  MnSymbolE8
   <9-10> MnSymbolE9
  <10-12> MnSymbolE10
  <12->   MnSymbolE12
}{}
\DeclareFontShape{OMX}{MnSymbolE}{b}{n}{
    <-6>  MnSymbolE-Bold5
   <6-7>  MnSymbolE-Bold6
   <7-8>  MnSymbolE-Bold7
   <8-9>  MnSymbolE-Bold8
   <9-10> MnSymbolE-Bold9
  <10-12> MnSymbolE-Bold10
  <12->   MnSymbolE-Bold12
}{}

\let\llangle\@undefined
\let\rrangle\@undefined
\DeclareMathDelimiter{\llangle}{\mathopen}%
                     {MnLargeSymbols}{'164}{MnLargeSymbols}{'164}
\DeclareMathDelimiter{\rrangle}{\mathclose}%
                     {MnLargeSymbols}{'171}{MnLargeSymbols}{'171}
\makeatother

\newtheorem{theorem}{Theorem}
\newtheorem{corollary}[theorem]{Corollary}
\newtheorem{lemma}[theorem]{Lemma}
\newtheorem{proposition}[theorem]{Proposition}

\newtheorem{ltheorem}{Theorem}

\newtheorem{lcorollary}[ltheorem]{Corollary}

\theoremstyle{remark}
\newtheorem{remark}[theorem]{Remark}

\newtheorem*{remarkintro}{Remark}

\theoremstyle{definition}
\newtheorem{definition}[theorem]{Definition}

\numberwithin{theorem}{section}
\numberwithin{equation}{section}

\allowdisplaybreaks

\title{Sparse domination implies convex body domination}

\author[Laukkarinen]{Aapo Laukkarinen}
\address{Aapo Laukkarinen \hfill\break\indent
Department of Mathematics and Systems Analysis \hfill\break\indent
Aalto University \hfill\break\indent
P.O. Box 11100 \hfill\break\indent
FI-00076 Aalto, Finland}
\email{aapo.laukkarinen@aalto.fi}

\author[Lorist]{Emiel Lorist}
\address{Emiel Lorist \hfill\break\indent
Delft Institute of Applied Mathematics \hfill\break\indent
Delft University of Technology \hfill\break\indent
P.O. Box 5031 \hfill\break\indent
2600 GA Delft, The Netherlands}
\email{e.lorist@tudelft.nl}


\thanks{E.~Lorist was financed by the Dutch Research Council (NWO) on the project ``The sparse revolution for stochastic partial differential equations'' with project number \href{https://doi.org/10.61686/ZGRMR99948}{VI.Veni.242.057}.}

\thanks{A.~Laukkarinen was supported by the Research
Council of Finland through grant 364208.}

\begin{document}
\begin{abstract}
We prove that sparse domination of a bilinear form implies convex body domination. More precisely, if a bilinear form admits an $(r,s)$-sparse bound, then its coordinate-wise extension to $\mathbb C^n$-valued functions admits an $(r,s)$-convex body sparse bound. The proof relies on a randomization argument. We establish the result both for sparse families in a fixed dyadic lattice and for sparse families of arbitrary cubes. As an application, we deduce sparse domination for iterated commutators, with the local oscillations of the symbol appearing in the sparse form.

\end{abstract}

\maketitle

\section{Introduction}
Weighted norm inequalities play a central role in real-variable harmonic analysis. The theory goes back to the foundational work of Muckenhoupt on the Hardy--Littlewood maximal operator \cite{Muc1972} and of Hunt, Muckenhoupt and Wheeden on the Hilbert transform \cite{HMW1973}, which identified the Muckenhoupt $A_p$-classes as the natural scale of weights for these operators. Coifman and Fefferman subsequently proved that every Calder\'on--Zygmund operator is bounded on $L^p(w)$ for all $p\in(1,\infty)$ and $w\in A_p$ \cite{CF1974}. A striking feature of the Muckenhoupt weighted theory is Rubio de Francia's extrapolation theorem \cite{RdF1984}: if an operator $T$ is bounded on $L^{p_0}(w)$ for \emph{some} $p_0\in[1,\infty)$ and  every $w\in A_{p_0}$, then $T$ is bounded on $L^p(w)$ for \emph{every} $p\in(1,\infty)$ and $w\in A_p$.

Once the qualitative weighted theory for most classical operators in harmonic analysis was settled, a central question became quantitative: \begin{quote}
    How does the operator norm depend on the weight characteristic $[w]_{A_p}$?
\end{quote} This question was first answered for the maximal operator by Buckley \cite{Buc1993}, but was particularly relevant for the Beurling--Ahlfors transform, since Astala, Iwaniec and Saksman \cite{AIS2001} had shown that a linear $A_2$-bound had implications for Beltrami equations. Petermichl and Volberg proved this estimate for the Beurling--Ahlfors transform \cite{PV2002}. Afterwards, Petermichl established the sharp dependence on the weight characteristic for the Hilbert and Riesz transforms \cite{Pet2007,Pet2008}. 
For general Calder\'on--Zygmund operators, by a sharp form of Rubio de Francia extrapolation \cite{DGPP2005}, it suffices to study the case $p=2$, for which the sharp bound was called the $A_2$-conjecture. This conjecture was ultimately established by Hyt\"onen \cite{Hyt2012}. Lerner subsequently found an alternative proof using his local mean oscillation decomposition \cite{Ler2013}, estimating the norm of a Calder\'on--Zygmund operator by the norm of a so-called sparse operator. These ideas were then developed by various authors into the domination of a wealth of operators by positive sparse operators or forms. These developments were dubbed the ``sparse revolution'' by Pereyra \cite{Per2019}, to which we refer for a thorough historical account. 

\medskip

Let us briefly recall the terminology. A collection of cubes $\mc S$ is called $\eta$-sparse for $\eta \in (0,1]$ if for every $Q\in\mc S$ there is a measurable set $E_Q\subseteq Q$ such that $\abs{E_Q}\geq\eta\abs{Q}$ and the sets $E_Q$ for $Q \in \mc{S}$ are pairwise disjoint. For an operator $T$, domination by an $(r,s)$-sparse form with $r,s\in[1,\infty)$ means that, for every $f,g\in L^\infty_c(\R^d)$, there exists an $\eta$-sparse collection $\mc S$ satisfying
\begin{equation}
    \absB{\int_{\R^d}{Tf}\,{g}}\lesssim \sum_{Q\in\mc S}|Q|  \nrm{f}_{\avgL^r(Q)}\nrm{g}_{\avgL^s(Q)} , \label{eq:sparseform}
\end{equation}
    where $\|f\|_{\avgL^r(Q)}\coloneqq\brb{\frac{1}{\abs{Q}} \int_Q \abs{f}^r}^{1/r}$ and the implicit constant is independent of $f$ and $g$. The same terminology applies to a general bilinear form $\Lambda\colon L^\infty_c(\R^d) \times L^\infty_c(\R^d) \to \C$, with the left-hand side of \eqref{eq:sparseform} replaced by $\abs{\Lambda(f,g)}$. 

Sparse domination has become a standard route to weighted norm inequalities, often yielding sharp dependence on the weight characteristic. The method separates the operator-specific domination step from the operator-independent weighted estimate for the resulting sparse form. Part of the success of sparse domination can be explained by the following converse observation: Let $T$ be an operator and suppose that, for some $p\in(1,\infty)$ and some increasing function $\phi\colon[1,\infty)\to[1,\infty)$, we have
\begin{equation*}
     \nrm{T}_{L^p(w)\to L^p(w)}\leq \phi\brb{[w]_{A_p}}\qquad w\in A_p.
\end{equation*}
Let $f,g \in L^\infty_c(\R^d)$ and $r,s \in (1,\infty)$.
Then $M_rf,M_{s}g \in A_1$ (see \cite{CR1980}), where $M_r$ denotes the rescaled Hardy--Littlewood maximal operator. Therefore $w:= (M_rf)^{1-p} \cdot M_{s}g \in A_p$ with $[w]_{A_p} \leq c_{d,p} (r')^{p-1}s'$, so we can estimate
    \begin{align}\label{eq:weighted}
        \absB{\int_{\R^d}{Tf}\,{g}}&\leq \nrm{Tf}_{L^p(w)} \nrm{g}_{L^{p'}(w')}\leq \phi(c_{d,p}(r')^{p-1}s')  \nrm{f}_{L^p(w)} \nrm{g}_{L^{p'}(w')},
    \end{align}
where $w' = w^{1-p'}$. Next note that by the Lebesgue differentiation theorem
\begin{align*}
    \nrm{f}_{L^p(w)}&=\brB{\int_{\R^d} \abs{f}^p \cdot (M_rf)^{1-p} \cdot M_{s}g}^{1/p} \leq \brB{\int_{\R^d} \abs{f} \cdot M_{s}g}^{1/p},\\
    \nrm{g}_{L^{p'}(w')}&=\brB{\int_{\R^d} \abs{g}^{p'} \cdot M_rf \cdot (M_{s}g)^{-\frac{1}{p-1}}}^{1/p'}\leq \brB{\int_{\R^d} M_r(f) \cdot  \abs{g}}^{1/p'}.
\end{align*}
Applying the standard sparse domination of the Hardy--Littlewood maximal operator to the two integrals above and plugging the resulting estimates into \eqref{eq:weighted}, we obtain a sparse collection $\mc{S}$ for which \eqref{eq:sparseform} holds.
This shows that, away from the endpoints, weighted norm estimates imply sparse form domination. Combined with the converse implication, this yields an equivalence between the two principles. This equivalence is only qualitative: passing from weighted estimates to sparse domination and back does not preserve the quantitative dependence on the weight characteristic. 

\medskip

One direction in which weighted theory has been extended is matrix-weighted estimates for vector-valued functions. In this setting we extend an operator $T$ (or the bilinear form $\Lambda$) acting on scalar-valued function to act on vector-valued function component-wise. The research on matrix-weighted estimates for singular integrals was initiated by Nazarov, Treil and Volberg \cite{NT1996,TV1997,Vol1997}. In particular they connected the boundedness of the Hilbert transform in the matrix-weighted $L^2$-space to some structural properties of multivariate stationary processes, and to boundedness of some Toeplitz operators. More recently, matrix weights have seen connections with systems of PDEs \cite{IM2019} and elliptic boundary value problems on compact manifolds \cite{BR2025}. 

In the light of the many sharp quantitative results in the scalar-weighted theory, it is natural to ask for the sharp dependence on the weight in a matrix-weighted inequality for Calder\'on-Zygmund operators. The difficulty for this question comes from the fact that matrices map vectors differently based on the direction of the vectors, and the absolute values inside the averages in sparse domination kill this information. Therefore, sparse domination does not seem to directly imply a matrix-weighted estimate. 
This difficulty was solved in \cite{NPTV2017} by Nazarov, Petermichl, Treil and Volberg via a technique that became known as convex body domination. The idea was to replace the averages in \eqref{eq:sparseform} with the convex sets of the form
\[
    \llangle\vec f\rrangle_{\avgL^r(Q)}\coloneqq\left\{\frac{1}{|Q|}\int_Q\vec f\phi\,\colon\,\|\phi\|_{\avgL^{r'}(Q)}\leq 1\right\}.
\]
Then \eqref{eq:sparseform} is replaced by
\[
    \absB{\sum_{j=1}^n\int_{\R^d}{Tf_j}\,{g_j}}\lesssim \sum_{Q\in\mc S}|Q|  \llangle\vec f\rrangle_{\avgL^r(Q)}\cdot\llangle\vec g\rrangle_{\avgL^s(Q)} ,
\]
where the Minkovski dot product in the summand is understood as a non-negative number, see Subsection \ref{Sect:ConvBod} for more details. In contrast to the scalar case, convex body domination yields a $3/2$-power dependence on the $A_2$-matrix weight characteristic, rather than the linear $A_2$ dependence available in the scalar case. Quite surprisingly, this $3/2$ dependence was shown to be sharp for the Hilbert transform in \cite{DPTV2024}. The failure of the linear bound in the matrix case implies that the sharp matrix-weighted Rubio de Francia extrapolation theorem proven in \cite{BC2026} does not determine the sharp dependence for Calder\'on-Zygmund operators for $p\neq 2$. Determining this dependence remains an open problem. As with sparse domination, convex body domination has since been extended to many settings beyond Calder\'on--Zygmund theory, see, for example, \cite{BBDPW2026,CLM2026,DHL2020,KN2024,Lau2025,Lau2026,MR2022}.

It was conjectured in \cite{NPTV2017} that sparse domination implies convex body domination. A partial result was achieved in \cite[Corollary 5.4]{Hyt2024}, where a typical generic estimate in sparse domination proofs was shown to imply convex body domination. This result and the ideas leading up to it have been used in many places, e.g. in \cite{Lau2025}, to prove convex body domination.  In essence, the argument consists of modifying the sparse domination proof so that the stopping time argument accounts for all of the components of the vector-valued input function and using a norm-to-convex-body lifting result that was already present in \cite{NPTV2017} and generalized in \cite{Hyt2024}. In this paper we provide a proof of the full general statement.

\begin{ltheorem}\label{Thm:SparseToConvexBodynondyadicintro}
Let $1\leq r,s<\infty$ with $\frac{1}{r}+\frac{1}{s}>1$. Let $\Lambda\colon L^\infty_c(\R^d) \times L^\infty_c(\R^d) \to \C$ be a bilinear form. Suppose that there is a constant $C_{\mathrm{Sp}}$  such that for all $f,g \in L^\infty_c(\R^d)$ 
 there is an $\eta$-sparse family of cubes $\mathcal{S}$ for which
    \[
        |\Lambda(f,g)|\leq C_{\mathrm{Sp}}\sum_{Q\in\mathcal{S}}|Q|\,\|f\|_{\avgL^r(Q)}\|g\|_{\avgL^s(Q)}.
    \]
    Then, for all $\vec f,\vec g \in L^\infty_c(\R^d,\C^n)$, there exists a $\frac{1}{2\cdot 3^d}$-sparse family of cubes $\mathcal{S}$ for which
    \[
        |\Lambda(\vec f,\vec g)|\lesssim_{d,r,s} n^{1+\frac{1}{r}+\frac{1}{s}}\eta^{-1}C_{\mathrm{Sp}}\sum_{Q\in\mathcal{S}}|Q|\llangle \vec f\rrangle_{\avgL^r(Q)}\cdot\llangle \vec g\rrangle_{\avgL^s(Q)}.
    \] 
\end{ltheorem}

Theorem \ref{Thm:SparseToConvexBodynondyadicintro} follows directly from Theorem \ref{Thm:SparseToConvexBodynondyadic} in the body of the paper. A version for dyadic cubes  with constants independent of the ambient dimension can be found in Theorem \ref{Thm:UniLocSparseToConvexBody}.
The key part of the proof is a randomization of the good part of the bilinear form, which eliminates the contribution of small-scale cubes. See Lemmas \ref{Lemma:KeyRandomization} and \ref{Lemma:AsymmetricRandomization} for the dyadic and non-dyadic versions, respectively. The rest of the argument is essentially already present in \cite{Hyt2024,NPTV2017}.

\begin{remarkintro}
We can combine Theorem \ref{Thm:SparseToConvexBodynondyadicintro} with the following three ingredients:
\begin{itemize}
    \item The aforementioned observation that weighted estimates imply sparse domination;
    \item Matrix-weighted estimates for convex body sparse forms, see \cite{NPTV2017} and \cite[Theorem 1.3]{KNV2024};
    \item Self-improvement of the scalar weights associated with a matrix $A_p$-weight and its dual (see \cite[Lemma 6.3]{Lau2025}).
\end{itemize}
This yields the following consequence: Let $T$ be a linear operator and suppose that there is a $p\in(1,\infty)$ and an  increasing function $\phi\colon[1,\infty)\to[1,\infty)$ such that for every $w\in A_p$ we have
\begin{equation*}
     \nrm{T}_{L^p(w)\to L^p(w)}\leq \phi\brb{[w]_{A_p}}.
\end{equation*} Then, for every $n\in\N$, there exists an increasing function $\psi\colon[1,\infty)\to[1,\infty)$ such that for every $n\times n$ matrix weight $W\in A_p$ we have
\[
    \nrm{T}_{L^p(W,\C^n)\to L^p(W,\C^n)} \leq \psi\brb{[W]_{A_p}},
\]
where $T$ acts componentwise. A detailed proof will be included in a forthcoming version of this paper.
\end{remarkintro}

Another direction in which sparse domination has been developed is the study of commutators. Pointwise sparse domination for first-order commutators of Calder\'on--Zygmund operators was established by Lerner, Ombrosi and Rivera-R\'ios \cite{LORR2017}. The method was subsequently extended to iterated commutators \cite{HLO2020, IFRR2020}, rough singular integral operators \cite{RR2018}, paraproducts \cite{FHF2023,HLS2025}, and a broad class of operators admitting $(r,s)$-sparse form bounds \cite{LLO2024}.
These results are obtained by returning to the scalar sparse domination argument for the underlying operator and repeating or modifying its stopping-time construction to explicitly keep track of the function $b$ from the commutator.

A different mechanism was developed in the context of convex body domination. It was shown that convex body domination of an operator implies a sparse domination of its commutators, see  \cite[Theorem 4]{IPRR2021} and \cite[Theorem 1.8]{IPT2022}. An analogous version for iterated multi-symbol commutators was recently obtained in \cite{IRSR2026}. An abstract formulation for generalized commutators was subsequently obtained in \cite[Lemma 7.1 and Example 7.6]{Hyt2024}; see also \cite[Proposition 7.1]{Lau2025} and \cite[Section 4]{CLM2026}. Combining this with Theorem \ref{Thm:SparseToConvexBodynondyadicintro} yields a direct implication from scalar sparse domination to sparse domination of commutators.

To state the resulting corollary, let $\Lambda\colon L^\infty_c(\R^d)\times L^\infty_c(\R^d)\to\C$ be a bilinear form and let $b\in L^\infty_{\loc}(\R^d)$. Set $\Lambda_b^0\coloneqq\Lambda$ and recursively define
\[
    \Lambda_b^m(f,g) \coloneqq  \Lambda_b^{m-1}(f,bg)-\Lambda_b^{m-1}(bf,g),
    \qquad m\geq1,
\]
or equivalently,
\[
\Lambda_b^m(f,g)=\sum_{k=0}^m(-1)^k\binom{m}{k}\Lambda(b^kf,b^{m-k}g),\qquad m\geq1.
\]
Note that if $\Lambda=\int_{\R^d}Tf\,g$ for a linear operator $T$,
then $\Lambda_b^m(f,g)=\int_{\R^d}T_b^mf\,g,$
where $T_b^m$ is the $m$th-order commutator, i.e. $T_b^0\coloneqq T$ and $T_b^m\coloneqq[b,T_b^{m-1}]$.

\begin{lcorollary}\label{Cor:SparseToCommutatorSparseintro}
Let $1\leq r,s<\infty$ with $\frac{1}{r}+\frac{1}{s}>1$. Let $\Lambda\colon L^\infty_c(\R^d) \times L^\infty_c(\R^d) \to \C$ be a bilinear form. Suppose that there is a constant $C_{\mathrm{Sp}}$  such that for all $f,g \in L^\infty_c(\R^d)$ 
 there is an $\eta$-sparse family of cubes $\mathcal{S}$ for which
    \[
        |\Lambda(f,g)|\leq C_{\mathrm{Sp}}\sum_{Q\in\mathcal{S}}|Q|\,\|f\|_{\avgL^r(Q)}\|g\|_{\avgL^s(Q)}.
    \]
    Then, for every $m\geq1$, $b\in L^\infty_{\loc}(\R^d)$ and $f,g\in L^\infty_c(\R^d)$, there exists a $\frac{1}{2\cdot 3^d}$-sparse family of cubes $\mathcal{S}$ such that, for every choice of constants $(c_Q)_{Q\in\mathcal{S}}\subset\C$,
\begin{align*}
|\Lambda_b^m(f,g)|
&\lesssim_{d,r,s,m}\eta^{-1}C_{\mathrm{Sp}}
\sum_{Q\in\mathcal{S}}|Q|\bigl(
\|(b-c_Q)^mf\|_{\avgL^r(Q)}\|g\|_{\avgL^s(Q)}+\|f\|_{\avgL^r(Q)}\|(b-c_Q)^mg\|_{\avgL^s(Q)}
\bigr).
\end{align*}
\end{lcorollary}
Corollary \ref{Cor:SparseToCommutatorSparseintro} follows from Theorem \ref{Thm:SparseToConvexBodynondyadicintro} by the same argument used to deduce Corollary \ref{Cor:SparseToCommutatorSparse} from Theorem \ref{Thm:SparseToConvexBody} in the body of the paper. Combining this corollary with the earlier observation that weighted bounds and sparse domination are equivalent, one finds that Muckenhoupt weighted estimates for an operator $T$ imply Bloom weighted estimates for the commutators $[b,T]$ with $b \in L^\infty_{\loc}(\R^d)$. We leave the precise statement for future work.

\medskip

This paper is organized as follows. In Section~\ref{sec:dyadic}, we treat sparse domination relative to a fixed dyadic lattice. We first introduce the required convex body formalism and then combine a stopping-time construction with the randomization lemma at the heart of the proof. This yields the dyadic version of our main theorem and, as a consequence, sparse bounds for iterated commutators. In the final section, we reduce sparse forms over arbitrary cubes to asymmetric dyadic sparse forms, extend the stopping-time and randomization arguments, and deduce Theorem \ref{Thm:SparseToConvexBodynondyadicintro}.

\section{Dyadic sparse domination}\label{sec:dyadic}
To isolate the main idea, we first prove that sparse domination with respect to a fixed dyadic lattice implies convex body domination with respect to the same lattice. In the next section, we extend this result to sparse forms of  arbitrary cubes.

\subsection{General convex bodies and the extension of bilinear forms}\label{Sect:ConvBod}
Let $X$ be a complex Banach space. For a vector $\vec{x} \in X^n$ we define the convex bodies $\llangle \vec x\rrangle_X$ by
\[
    \llangle \vec x\rrangle_X\coloneqq \{\langle \vec x,x^*\rangle\,\colon\, x^*\in \Bar B_{X^*}\}\subset \C^n,
\]
where $X^*$ is the dual space of $X$ and $\langle \vec x,x^*\rangle\coloneqq(\langle x_i,x^*\rangle)_{i=1}^n$. It is clear that $ \llangle \vec x\rrangle_X$ is absolutely convex and bounded. Furthermore, it is closed and hence compact, see \cite[Lemma 2.3]{Hyt2024}.
We denote
\[
    \llangle \vec x\rrangle_X\cdot\llangle \vec y\rrangle_Y\coloneqq \sup\bigl\{|\vec a \cdot\vec b|\,\colon\,\vec a\in\llangle \vec x\rrangle_X, \vec b\in \llangle \vec y\rrangle_Y\bigr\},
\]
where we use the bilinear dot product $\vec a \cdot \vec b = \sum_{j=1}^n a_jb_j$. 

We will use the following convex body lifting, which has a sharper  dependence on $n$ compared to the more widely used John ellipsoid approach, cf. \cite[Lemma 4.1]{Hyt2024}.

\begin{lemma}[Auerbach reduction]\label{lem:auerbach}
Let \(X\) and \(Y\) be Banach spaces and let \(( e_j)_{j=1}^n\) be the standard basis of
\(\mathbb C^n\).
For every \(\vec x\in X^n\) there exists an invertible linear transformation $A_{\vec{x}} \colon \C^n\to \C^n$ such that for all \(\vec y\in Y^n\), defining 
\[x_j':=A_{\vec{x}}\vec x\cdot e_j,
\qquad
y_j':=A^{-\top}_{\vec{x}}\vec y\cdot e_j,
\qquad j=1,\ldots,n,\]
we have 
\[
\sum_{j=1}^n\|x_j'\|_X\|y_j'\|_Y\leq
n\,\llangle \vec x\rrangle_X\cdot\llangle \vec y\rrangle_Y.\]
\end{lemma}

\begin{proof}
Set $U:=\operatorname{span} \, \llangle \vec x\rrangle_X$, which we norm using the Minkowski functional of $\llangle \vec x\rrangle_X$, i.e.
$$
\nrm{u}_U := \inf\cbrace{\lambda >0:\lambda^{-1} u \in \llangle \vec x\rrangle_X}.
$$
Then $\llangle \vec x\rrangle_X$ is the closed unit ball in $U$. Denoting $m:=\dim U\leq n$ we know that by the Auerbach lemma (see, e.g., \cite[Lemma II.E.11]{Woj91}) there exists a biorthogonal basis $(u_j,u_j^*)_{j=1}^m$ in $U \times U^*$ with $\nrm{u_j}_U= \nrm{u_j^*}_{U^*} =1$ for $j=1,\ldots,m$. Extend \((u_j)_{j=1}^m\) to a basis \((u_j)_{j=1}^n\) of \(\mathbb C^n\), and define the invertible transformation \(A_{\vec x} \colon \C^n \to \C^n\) by
\[
A_{\vec x}u_j:= e_j,\qquad j=1,\ldots,n.
\]
For \(j\leq m\), we have
\[\|x_j'\|_X=\sup_{\nrm{x^*}_{X^*}\leq 1}\absb{\brb{A_{\vec x}\langle \vec x,x^*\rangle}   \cdot e_j}
=\sup_{\nrm{x^*}_{X^*}\leq 1}\absb{u_j^*\brb{\langle \vec x,x^*\rangle}}\leq1.
\]
On the other hand, we have
$y_j'=A_{\vec x}^{-\top}\vec y\cdot e_j=\vec y\cdot A_{\vec x}^{-1}  e_j=\vec y\cdot u_j.$ and thus, since \(u_j\in\llangle\vec x\rrangle_X\), we obtain
\[
\|y_j'\|_Y=\sup_{\nrm{y^*}_{Y^*}\leq1}\absb{u_j\cdot\langle\vec y,y^*\rangle}\leq
\llangle\vec x\rrangle_X\cdot\llangle\vec y\rrangle_Y.
\]
Moreover, since \(\llangle \vec x\rrangle_X\subset U\), all coordinates of
\(\vec{x}'=A_{\vec x}\vec{x}\) with index \(j>m\) vanish. Therefore
\[
\sum_{j=1}^n\|x_j'\|_X\|y_j'\|_Y=\sum_{j=1}^m\|x_j'\|_X\|y_j'\|_Y\leq m\,\llangle\vec x\rrangle_X \cdot\llangle\vec y\rrangle_Y\leq n\,\llangle\vec x\rrangle_X\cdot\llangle\vec y\rrangle_Y,
\]
which proves the result.
\end{proof}

We extend a bilinear form $\Lambda\colon X\times Y\to \C$ to a form $X^n\times Y^n \to \C$ by 
\[
    \Lambda(\vec x,\vec y)\coloneqq \sum_{j=1}^n\Lambda(\vec x\cdot e_j,\vec y\cdot e_j),
\]
where $(e_j)_{j=1}^n$ is the standard basis of $\C^n$. It is simple to check that for any invertible linear transformation $A \colon \C^n \to \C^n$ we have \[
    \Lambda(A\vec x, \vec y)=\Lambda(\vec x,A^{\top}\vec y),
\]
and therefore $\Lambda(A\vec x,A^{-\top}\vec y)=\Lambda(\vec x,\vec y).$

\begin{remark}\label{rem:sesq}
    If the Banach space $Y$ admits a natural conjugation, the preceding bilinear extension can be related to the usual, basis-independent, extension of a sesquilinear form. Indeed,  $\Lambda_{\mathrm{ses}}\colon X\times Y\to\C$ given by
    \[
        \Lambda_{\mathrm{ses}}(x,y)\coloneqq\Lambda(x,\overline y),\qquad x\in X,\ y\in Y,
    \]
    is sesquilinear, and its usual extension is
    \[
        \Lambda_{\mathrm{ses}}(\vec x,\vec z)\coloneqq\sum_{j=1}^n\Lambda_{\mathrm{ses}} \bigl((\vec x,v_j)_{\C^n},(\vec z,v_j)_{\C^n}\bigr),
    \]
    where $(v_j)_{j=1}^n$ is any orthonormal basis of $\C^n$. This extension is independent of the choice of orthonormal basis. Taking the standard basis, which is real, yields $\Lambda(\vec x,\vec y)=\Lambda_{\mathrm{ses}}(\vec x,{\overline{\vec y}}).$ Thus our bilinear extension is invariant under changes of real orthonormal basis, but in general not under arbitrary changes of orthonormal basis.
\end{remark}

\subsection{Stopping time argument}
By a cube we mean a cube whose sides are parallel to the coordinate axes. 
Given a cube $Q$, we denote its side length by $\ell(Q)$ and let $\mathcal{D}_k(Q)$ be the collection of cubes obtained by subdividing $Q$ into $2^{kd}$ congruent subcubes with disjoint interiors. Let $\mathcal{D}(Q)\coloneqq\bigcup_{k=0}^\infty\mathcal{D}_k(Q)$. 

For a cube $Q$ and $1\leq p<\infty$, we equip the local normalized Lebesgue space $\avgL^p(Q)$ with the norm
\[
    \|f\|_{\avgL^p(Q)}\coloneqq\brB{\frac{1}{\abs{Q}} \int_Q \abs{f}^p}^{1/p}.
\]
The localized dyadic $p$-maximal operator is defined by
\[
M_p^{\mathcal D(Q)}f(x)\coloneqq\sup_{\substack{R\in\mathcal D(Q):x\in R}}\|f\|_{\avgL^p(R)},\qquad x\in Q,
\]
which satisfies the weak-type estimate
\[
\sup_{\lambda>0}\lambda \absb{\bigl\{x\in Q:M_p^{\mathcal D(Q)}f(x)>\lambda\bigr\}}^{1/p}
\leq\|f\|_{L^p(Q)}.
\]

In the following lemma we record the standard parallel stopping-time argument. We provide a proof for completeness.
\begin{lemma}\label{lemma:multistoptime}
    Let $1\leq r,s<\infty$ and let $Q$ be a cube. Let $f_j\in \avgL^r(Q)$ and $g_j\in \avgL^s(Q)$ for $j=1,\dots, n$. Then there exists a disjoint collection $\mathcal{F}_Q\subset \mathcal{D}(Q)$ such that
    \begin{equation*}
        \sum_{R\in\mathcal{F}_Q}|R|\leq \tfrac{1}{2}|Q|,
    \end{equation*}
    and if $P\in\mathcal{D}(Q)$ is not contained in any $R\in \mathcal{F}_Q$, then for $j=1,\ldots,n$
    \begin{equation*}
        \|f_j\|_{\avgL^r(P)}\leq (4n)^\frac{1}{r}\|f_j\|_{\avgL^r(Q)}, \qquad \|g_j\|_{\avgL^s(P)}\leq (4n)^\frac{1}{s}\|g_j\|_{\avgL^s(Q)} .
    \end{equation*}
\end{lemma}

\begin{proof}
    Let $\mathcal{F}_Q$ be dyadic maximal subcubes of $Q$ that satisfy 
    \[
        \|f_j\|_{\avgL^r(R)}> (4n)^\frac{1}{r}\|f_j\|_{\avgL^r(Q)} \qquad \text{or}\qquad \|g_j\|_{\avgL^s(R)}> (4n)^\frac{1}{s}\|g_j\|_{\avgL^s(Q)}
    \]
    for some $1\leq j\leq n$. Then we have
    \begin{equation}\label{Eq:CubesToLevelSets}
        \sum_{R\in\mathcal{F}_Q}|R|\leq \sum_{j=1}^n|E_j|+\sum_{j=1}^n|\Tilde{E}_j|,
    \end{equation}
    where
    \begin{align*}
            E_j&\coloneqq \{M^{\mathcal D(Q)}_{r}f_j>(4n)^\frac{1}{r}\|f_j\|_{\avgL^r(Q)}\}, \\ \widetilde{E}_j&\coloneqq \{M^{\mathcal D(Q)}_{s}g_j>(4n)^\frac{1}{s}\|g_j\|_{\avgL^s(Q)}\}.
    \end{align*}
By the weak boundedness of the dyadic maximal function we have \[\max\cbraceb{|E_j|,|\widetilde{E}_j}\leq \tfrac{1}{4n}\abs{Q}, \qquad 1\leq j\leq n.\]  Plugging this into \eqref{Eq:CubesToLevelSets} yields the first assertion. The second assertion follows directly from maximality.
\end{proof}

\subsection{Randomization of the good part of the bilinear form}
The following  lemma provides the key ingredient that was previously missing from the proof of the general statement that sparse domination implies convex body domination. 
\begin{lemma}\label{Lemma:KeyRandomization}
    Let $p \in (1,\infty)$ and let $\Lambda\colon L^p(\R^d) \times L^{p'}(\R^d) \to \C$ be a continuous bilinear form. Fix a cube $Q$, let $\mathcal{F}_Q\subset \mathcal{D}(Q)$ be a collection of disjoint cubes and let $\{\varepsilon_R\}_{R\in \mathcal{F}_Q}$ be independent Bernoulli random variables with $\P(\cbrace{\varepsilon_R=0})= \P(\cbrace{\varepsilon_R=1})=\frac12$. Set $G_Q\coloneqq Q\setminus\bigcup_{R\in\mathcal{F}_Q}R$ and
    \[
        m_\varepsilon:=\ind_{G_Q}+2\sum_{R\in\mathcal{F}_Q}\ind_R\varepsilon_R,\qquad n_\varepsilon:=\ind_{G_Q}+2\sum_{R\in\mathcal{F}_Q}\ind_R(1-\varepsilon_R).
    \]
    Then for $f \in L^p(Q)$ and $g \in L^{p'}(Q)$ we have
    \begin{equation}\label{Eq:randomize}
        \E \left[\Lambda(m_\varepsilon f,n_\varepsilon g)\right]=\Lambda(f,g)- \sum_{R\in\mathcal{F}_Q}\Lambda(\ind_Rf,\ind_Rg).
    \end{equation}
    Moreover, for every $R\in\mathcal{F}_Q$, either $m_\varepsilon$ or $n_\varepsilon$ vanishes on $R$.
\end{lemma}

\begin{proof} 
    Since $\Lambda$ is continuous and bilinear, we have
    \begin{align*}          \Lambda(m_\varepsilon f,n_\varepsilon g) =\Lambda(\ind_{G_Q}f,\ind_{G_Q}g)&+2\sum_{R\in\mathcal{F}_Q}\varepsilon_R \Lambda(\ind_{R}f,\ind_{G_Q}g) \\&+ 2\sum_{R'\in\mathcal{F}_Q}(1-\varepsilon_{R'}) \Lambda(\ind_{G_Q}f,\ind_{R'}g)\\&+4\sum_{R\in\mathcal{F}_Q}\sum_{R'\in\mathcal{F}_Q}\varepsilon_R(1-\varepsilon_{R'}) \Lambda(\ind_{R}f,\ind_{R'}g).
    \end{align*}
    We note that $4\varepsilon_R(1-\varepsilon_{R})=0$, and all of the other coefficients have expectation $1$. Therefore, 
    \begin{align*}
        \E \left[\Lambda(m_\varepsilon f,n_\varepsilon g)\right]=\Lambda(\ind_{G_Q}f,\ind_{G_Q}g)+\sum_{R\in\mathcal{F}_Q}\Lambda(\ind_{R}f,\ind_{G_Q}g) &+ \sum_{R'\in\mathcal{F}_Q} \Lambda(\ind_{G_Q}f,\ind_{R'}g)\\&+\sum_{\substack{R,R'\in\mathcal{F}_Q\\R\neq R'}}\Lambda(\ind_{R}f,\ind_{R'}g).
    \end{align*}
    By continuity and bilinearity, the last expression is equal to $\Lambda(f,g)- \sum_{R\in\mathcal{F}_Q}\Lambda(\ind_Rf,\ind_Rg)$. 
    The final assertion follows from the facts that if $x\in R\in\mathcal{F}_Q$, then $$m_\varepsilon(x)n_\varepsilon(x)=4(1-\varepsilon_R)\varepsilon_R=0$$
    and that $m_\varepsilon$ and $n_\varepsilon$ are constant on $R$.
\end{proof}

\subsection{Proof of the main dyadic result}
To pass from the preceding local arguments to a global result, we use the following notion of a dyadic lattice from \cite{LN2019}.

\begin{definition}\label{Def:DyadicLattice}
A collection of cubes $\mathcal D$ is called a dyadic lattice if it satisfies the following properties:
\begin{enumerate}
    \item If $Q\in\mathcal D$ and $Q'\in\mathcal D(Q)$, then $Q'\in\mathcal D$.
    \item If $Q',Q''\in\mathcal D$, then there exists a $Q\in\mathcal D$ such that $Q',Q''\in\mathcal D(Q)$.
    \item Every compact subset of $\R^d$ is contained in some cube from $\mathcal D$.
\end{enumerate}
\end{definition}

Throughout the rest of this section, we will consider the dyadic lattice $\mc{D}$ to be fixed. 
We start by proving that uniform local sparse domination implies convex body domination.

\begin{theorem}\label{Thm:UniLocSparseToConvexBody}
    Let $1\leq r,s<\infty$ and $1<p<\infty$. Let $\Lambda\colon L^p(\R^d)\times L^{p'}(\R^d)\to\C$ be a continuous bilinear form. Suppose that there is a constant $C_{\mathrm{Sp}}$  such that for every $Q\in \mc{D}$ and  $f,g\in L^\infty(Q)$ there is an $\eta$-sparse family $\mathcal{S}_Q\subset\mathcal{D}(Q)$ for which 
    \[
        |\Lambda(f,g)|\leq C_{\mathrm{Sp}}\sum_{P\in\mathcal{S}_Q}|P|\,\|f\|_{\avgL^r(P)}\|g\|_{\avgL^s(P)}.
    \]
    Then, for all $\vec f,\vec g \in L^\infty_c(\R^d,\C^n)$, there exists a $\frac{1}{2}$-sparse family $\mathcal{S}\subset\mathcal{D}$ for which
    \[
        |\Lambda(\vec f,\vec g)|\leq (4n)^{1+\frac{1}{r}+\frac{1}{s}}\eta^{-1}C_{\mathrm{Sp}}\sum_{P\in\mathcal{S}}|P|\llangle \vec f\rrangle_{\avgL^r(P)}\cdot\llangle \vec g\rrangle_{\avgL^s(P)}.
    \] 
\end{theorem}
\begin{proof}
    Fix a cube $Q \in \mc{D}$ and let $A\colon \C^n\to \C^n$ be the linear transformation of Lemma \ref{lem:auerbach} with $(\vec f,\avgL^r(Q))$ in place of $(\vec x,X)$, and denote $f_j^Q\coloneqq A\vec f\cdot e_j$ and $g_j^Q\coloneqq A^{-\top}\vec g\cdot e_j$.
    Let $\mathcal{F}_Q$ be the collection given by Lemma \ref{lemma:multistoptime} using $\vec f^Q$ and $\vec g^Q$. Defining $m_\varepsilon$ and $n_\varepsilon$ as in  Lemma \ref{Lemma:KeyRandomization}, we have by the randomization identity \eqref{Eq:randomize} that
    \[
        \sum_{j=1}^n\Big[\Lambda(\ind_Qf_j^Q,\ind_Qg_j^Q)- \sum_{R\in\mathcal{F}_Q}\Lambda(\ind_Rf_j^Q,\ind_Rg_j^Q)\Big]= \sum_{j=1}^n\E \left[\Lambda(m_\varepsilon f_j^Q,n_\varepsilon g_j^Q)\right].
    \]
    Denote the probability space underlying $(\varepsilon_R)_{R \in \mc{F}_Q}$ by $\Omega$. For each $1\leq j\leq n$ and fixed $\omega \in \Omega$, applying the assumption gives us an $\eta$-sparse collection $\mathcal{S}_{Q,j,\omega}$ such that
    \[
        |\Lambda(m_\varepsilon(\omega) f_j^Q,n_\varepsilon(\omega) g_j^Q)|\leq C_{\mathrm{Sp}}\sum_{P\in\mathcal{S}_{Q,j,\omega}}|P|\,\|m_\varepsilon(\omega) f_j^Q\|_{\avgL^r(P)}\|n_\varepsilon(\omega) g_j^Q\|_{\avgL^s(P)}.
    \]
    The second conclusion of Lemma \ref{Lemma:KeyRandomization} implies that if $P\subset R\in\mathcal{F}_Q$, then the product in the above summand will vanish. Moreover, we have $0\leq m_\varepsilon(\omega),n_\varepsilon(\omega)\leq 2$. These considerations together with the conclusion of Lemma \ref{lemma:multistoptime} and the $\eta$-sparseness of $\mathcal{S}_{Q,j,\omega}$ imply that
    \begin{align*}
        \sum_{P\in\mathcal{S}_{Q,j,\omega}}|P|\,\|m_\varepsilon(\omega) f_j^Q\|_{\avgL^r(P)}\|n_\varepsilon(\omega) g_j^Q\|_{\avgL^s(P)}&\leq4\sum_{\substack{P\in\mathcal{S}_{Q,j,\omega}:\\P\not\subset R\text{ for all }R\in\mathcal{F}_Q}}|P|\,\| f_j^Q\|_{\avgL^r(P)}\| g_j^Q\|_{\avgL^s(P)}\\&\leq 4(4n)^{\frac{1}{r}+\frac{1}{s}} \eta^{-1}\abs{Q}\|f_j^Q\|_{\avgL^r(Q)}\|g_j^Q\|_{\avgL^s(Q)}.
    \end{align*}
    Note that the right-hand side is independent of $\omega$, so the same estimate holds after taking the expectation on the left-hand side. Moreover, the convex body lifting from Lemma \ref{lem:auerbach} yields
    \[
        \sum_{j=1}^n\|f_j^Q\|_{\avgL^r(Q)}\|g_j^Q\|_{\avgL^s(Q)}\leq n\,\llangle \vec f\rrangle_{\avgL^r(Q)}\cdot\llangle \vec g\rrangle_{\avgL^s(Q)}.
    \]
    Combining everything so far, we have proven that
    \[
        |\Lambda(\ind_Q\vec f,\ind_Q\vec g)- \sum_{R\in\mathcal{F}_Q}\Lambda(\ind_R\vec f,\ind_R\vec g)|\leq (4n)^{1+\frac{1}{r}+\frac{1}{s}}\eta^{-1}C_{\mathrm Sp}\,|Q|\llangle \vec f\rrangle_{\avgL^r(Q)}\cdot\llangle \vec g\rrangle_{\avgL^s(Q)}.
    \]
    We choose a cube $Q_0\in\mc D$ that satisfies $Q_0\supset\supp \vec f\cup\supp\vec g$. Then $\Lambda(\vec f,\vec g)=\Lambda(\ind_{Q_0}\vec f,\ind_{Q_0}\vec g)$ and 
    a standard iteration argument  gives us the desired conclusion. 
\end{proof}

Next we will show that for the relevant indices for weighted estimates (cf.  \cite{BFP2016}), the condition of standard sparse domination implies the assumption of the previous theorem and hence convex body domination. The proof is a simple tail estimate.

\begin{proposition}\label{Prop:GlobalToUniformLocal}
 Let $1\leq r,s<\infty$ with $\frac{1}{r}+\frac{1}{s}>1$. Let $\Lambda\colon  L^\infty_c(\R^d) \times L^\infty_c(\R^d) \to \C$ be a bilinear form. Suppose that there is a constant $C_{\mathrm{Sp}}$  such that for all $f,g \in L^\infty_c(\R^d)$ 
 there is an $\eta$-sparse family $\mathcal{S}\subset\mathcal{D}$ for which
    \[
        |\Lambda(f,g)|\leq C_{\mathrm{Sp}}\sum_{P\in\mathcal{S}}|P|\,\|f\|_{\avgL^r(P)}\|g\|_{\avgL^s(P)}.
    \]
    Then for every $Q\in\mathcal{D}$ and all $f,g \in L^\infty(Q)$  there is an $\frac{\eta}{2}$-sparse family $\mathcal{S}_Q\subset\mathcal{D}(Q)$ for which 
    \[
        |\Lambda(f,g)|\lesssim_{d,r,s}
        C_{\mathrm{Sp}}\sum_{P\in\mathcal{S}_Q}|P|\,\|f\|_{\avgL^r(P)}\|g\|_{\avgL^s(P)}.
    \]
\end{proposition}
\begin{proof}
    Fix $Q\in\mathcal{D}$ and let $f,g \in L^\infty(Q)$, which we extend by zero to elements of $L^\infty_c(\R^d)$. Applying the assumption gives an $\eta$-sparse collection $\mathcal{S}\subset \mathcal{D}$ such that
    \[
        |\Lambda(f,g)|\leq C_{\mathrm{Sp}}\sum_{P\in\mathcal{S}}|P|\,\|f\|_{\avgL^r(P)}\|g\|_{\avgL^s(P)}.
    \]
    For $P\in \mathcal S$ with $P \supsetneq Q$, let $\kappa\in\N$ such that    
    $\frac{\ell(Q)}{\ell(P)}=2^{-\kappa}$ and compute 
    \[
        |P|\|f\|_{\avgL^r(P)}\|g\|_{\avgL^s(P)}=\left(\frac{|Q|}{|P|}\right)^{\frac{1}{r}+\frac{1}{s}-1}|Q|\,\|f\|_{\avgL^r(Q)}\|g\|_{\avgL^s(Q)}=2^{\kappa d(1-\frac{1}{r}-\frac{1}{s})}|Q|\,\|f\|_{\avgL^r(Q)}\|g\|_{\avgL^s(Q)}.
    \]
    Therefore,
    \begin{align*}
        |\Lambda(f,g)|\leq C_{Sp}\sum_{\substack{P\in\mathcal{S}\\P\subset Q}}|P|\,\|f\|_{\avgL^r(P)}\|g\|_{\avgL^s(P)} + C_{\mathrm Sp}|Q|\,\|f\|_{\avgL^r(Q)}\|g\|_{\avgL^s(Q)}\sum_{\kappa=1}^\infty 2^{\kappa d(1-\frac{1}{r}-\frac{1}{s})}
    \end{align*}
    The geometric series converges since $\frac{1}{r}+\frac{1}{s}>1$, and the collection \[\mathcal{S}_Q\coloneqq \{P\in\mathcal{S}\,\colon\,P\subset Q\}\cup\{Q\}\] is $\frac{\eta}{2}$-sparse. This concludes the proof.
\end{proof}

Recall that, for every $p\in(r,s')$, sparseness, H\"older's inequality and the strong-type boundedness of the dyadic maximal operators yield
\[
\sum_{Q\in\mathcal S}|Q|\,\|f\|_{\avgL^r(Q)}\|g\|_{\avgL^s(Q)}
\leq\eta^{-1}\|M_r^{\mc D}f\|_{L^p(\R^d)}\|M_s^{\mc D}g\|_{L^{p'}(\R^d)}
\lesssim_{p,r,s}\eta^{-1}\|f\|_{L^p(\R^d)}\|g\|_{L^{p'}(\R^d)}.
\]
Therefore, every bilinear form satisfying sparse domination extends uniquely to a continuous bilinear form on $L^p(\R^d)\times L^{p'}(\R^d)$. Hence, combining Theorem \ref{Thm:UniLocSparseToConvexBody} and Proposition \ref{Prop:GlobalToUniformLocal}, we obtain our main dyadic result.

\begin{theorem}\label{Thm:SparseToConvexBody}
Let $1\leq r,s<\infty$ with $\frac{1}{r}+\frac{1}{s}>1$. Let $\Lambda\colon L^\infty_c(\R^d) \times L^\infty_c(\R^d) \to \C$ be a bilinear form. Suppose that there is a constant $C_{\mathrm{Sp}}$  such that for all $f,g \in L^\infty_c(\R^d)$ 
 there is an $\eta$-sparse family $\mathcal{S}\subset\mathcal{D}$ for which
    \[
        |\Lambda(f,g)|\leq C_{\mathrm{Sp}}\sum_{P\in\mathcal{S}}|P|\,\|f\|_{\avgL^r(P)}\|g\|_{\avgL^s(P)}.
    \]
    Then for all $\vec f,\vec g \in L^\infty_c(\R^d,\C^n)$, there exists a $\frac{1}{2}$-sparse family $\mathcal{S}\subset\mathcal{D}$ for which
    \[
        |\Lambda(\vec f,\vec g)|\lesssim_{d,r,s} n^{1+\frac{1}{r}+\frac{1}{s}}\eta^{-1}C_{\mathrm{Sp}}\sum_{P\in\mathcal{S}}|P|\llangle \vec f\rrangle_{\avgL^r(P)}\cdot\llangle \vec g\rrangle_{\avgL^s(P)}.
    \] 
\end{theorem}

 As discussed in the introduction, convex body domination yields sparse domination for iterated commutators, with the oscillations of the symbols appearing in the sparse form. Combining this observation with Theorem \ref{Thm:SparseToConvexBody} shows that sparse domination implies sparse domination of iterated commutators. We state it for a repeated symbol, although the same argument allows a different symbol at each commutation step.

\begin{corollary}\label{Cor:SparseToCommutatorSparse}
Let $1\leq r,s<\infty$ with $\frac{1}{r}+\frac{1}{s}>1$. Let $\Lambda\colon L^\infty_c(\R^d)\times L^\infty_c(\R^d)\to\C$ be a bilinear form. Suppose that there is a constant $C_{\mathrm{Sp}}$ such that for all $f,g\in L^\infty_c(\R^d)$ there is an $\eta$-sparse family $\mathcal{S}\subset\mathcal{D}$ for which
\[
|\Lambda(f,g)|\leq C_{\mathrm{Sp}}\sum_{Q\in\mathcal{S}}|Q|\,\|f\|_{\avgL^r(Q)}\|g\|_{\avgL^s(Q)}.
\]
Then, for every $m\geq1$, $b\in L^\infty_{\loc}(\R^d)$ and $f,g\in L^\infty_c(\R^d)$, there exists a $\frac{1}{2}$-sparse family $\mathcal{S}\subset\mathcal{D}$ such that, for every choice of constants $(c_Q)_{Q\in\mathcal{S}}\subset\C$,
\begin{align*}
|\Lambda_b^m(f,g)|&\lesssim_{d,r,s,m}\eta^{-1}C_{\mathrm{Sp}}
\sum_{Q\in\mathcal{S}}|Q|\bigl(\|(b-c_Q)^mf\|_{\avgL^r(Q)}\|g\|_{\avgL^s(Q)}+\|f\|_{\avgL^r(Q)}\|(b-c_Q)^mg\|_{\avgL^s(Q)}\bigr).
\end{align*}
\end{corollary}
\begin{proof}
By \cite[Lemma 7.1 and Example 7.6]{Hyt2024}, applied to the convex body domination of order $n=m+1$ in Theorem \ref{Thm:SparseToConvexBody}, there exists a $\frac{1}{2}$-sparse family $\mathcal{S}\subset\mathcal{D}$ such that
\[
|\Lambda_b^m(f,g)|\lesssim_{d,r,s,m}\eta^{-1}C_{\mathrm{Sp}}\sum_{Q\in\mathcal{S}}|Q|\llangle\vec b_mf\rrangle_{\avgL^r(Q)}\cdot\llangle\vec a_mg\rrangle_{\avgL^s(Q)},
\]
where
$\vec a_m\coloneqq\brb{b^{m-k}}_{k=0}^m$ and $\vec b_m\coloneqq\brb{(-1)^k\binom{m}{k}b^k}_{k=0}^m$ and thus
\[
\vec a_m(x)\cdot\vec b_m(y)=(b(x)-b(y))^m.
\]
Therefore,  the pointwise estimate
\[
|b(x)-b(y)|^m\leq2^{m-1}\bigl(|b(x)-c_Q|^m+|b(y)-c_Q|^m\bigr), \qquad Q \in \mc{S}
\]
gives the desired conclusion.
\end{proof}

\section{Sparse domination with arbitrary cubes}
For bilinear forms that are not adapted to a fixed dyadic lattice $\mc{D}$, one often obtains sparse domination of the form
\[
    \abs{\Lambda(f,g)}\leq C_{\mathrm{Sp}}\sum_{R\in\mathcal{S}}|R| \,\|f\|_{\avgL^r(R)}\|g\|_{\avgL^s(R)},
\]
where $\mathcal{S}$ is an $\eta$-sparse family of arbitrary cubes. The main idea from the preceding section is applicable to such forms, but some additional geometric bookkeeping is required.

\subsection{Reduction to an asymmetric dyadic sparse form}

We first show that a sparse form over arbitrary cubes can be replaced by a sparse form over cubes from a fixed dyadic lattice $\mc D$, at the cost of dilating the cubes in the $f$-average. For a cube $Q$, we denote by $3Q$ the cube with the same center as $Q$ and side length $3\ell(Q)$.

\begin{lemma}\label{Lemma:OneDyadicLattice}
Let $1\leq r,s<\infty$. If $\mc S$ is an $\eta$-sparse family of arbitrary cubes, then for every $f,g\in L^\infty_c(\R^d)$ there exists an $\frac{\eta}{6^d}$-sparse family $\mc T\subset\mc D$ such that
\[
    \sum_{Q\in\mc S}|Q|\,\|f\|_{\avgL^r(Q)}\|g\|_{\avgL^s(Q)}
    \lesssim_d\eta^{-1}\sum_{R\in\mc T}|R|\,\|f\|_{\avgL^r(3R)}\|g\|_{\avgL^s(R)}.
\]
\end{lemma}

\begin{proof}
For  $Q\in\mc S$ let $\mc F_Q\subseteq \mc{D}$ be the collection of cubes with $\ell(Q)\leq \ell(R) <2\ell(Q)$ and $|R\cap Q|>0$. Then $\mc F_Q$ contains at most $2^d$ cubes, covers $Q$ and for  every $R\in\mc F_Q$ we have  $Q\subset3R$. Choose $R_Q\in\mc F_Q$ such that
\[
    \|g\|_{\avgL^s(R_Q)}=\max_{R\in\mc F_Q}\|g\|_{\avgL^s(R)}.
\]
Since $|R_Q|\leq 2^d|Q|$ and $Q\subset3R_Q$, we have
\begin{align*}
        \|f\|_{\avgL^r(Q)}^r&\leq6^{d}\|f\|_{\avgL^r(3R_Q)}^r\\
    \|g\|_{\avgL^s(Q)}^s
    &\leq\frac{1}{|Q|}\sum_{R\in\mc F_Q}\int_R|g|^s
    \leq4^d\|g\|_{\avgL^s(R_Q)}^s.
\end{align*}

Set $\mc T\coloneqq\{R_Q:Q\in\mc S\}$. For $Q \in \mc{S}$ let $E_Q$ be the pairwise disjoint measurable sets such that $E_Q\subset Q$ and $|E_Q|\geq\eta|Q|$. Then we have
\[
    \sum_{\substack{Q\in\mc S\\R_Q=R}}|Q|
    \leq\eta^{-1}\sum_{\substack{Q\in\mc S\\R_Q=R}}|E_Q|
    \leq\eta^{-1}|3R|
    =3^d\eta^{-1}|R|,
\]
and therefore
\begin{align*}
    \sum_{Q\in\mc S}|Q|\,\|f\|_{\avgL^r(Q)}\|g\|_{\avgL^s(Q)}
    &\leq6^{d/r}4^{d/s}\sum_{R\in\mc T}\sum_{\substack{Q\in\mc S\\R_Q=R}}|Q|\,\|f\|_{\avgL^r(3R)}\|g\|_{\avgL^s(R)}\\
    &\lesssim_d \eta^{-1}\sum_{R\in\mc T}|R|\,\|f\|_{\avgL^r(3R)}\|g\|_{\avgL^s(R)}.
\end{align*}

Finally, to show sparseness of $\mc{T}$, for each $R\in\mc T$, choose one $Q_R\in\mc S$ such that $R_{Q_R}=R$. Note that if $Q_R=Q_{R'}$ for $R,R' \in \mc{T}$, then $R=R_{Q_{R}}=R_{Q_{R'}}=R'$, so the $Q_R$'s are distinct. Therefore, for every $P\in\mc D$, we obtain
\[
    \sum_{\substack{R\in\mc T\\R\subset P}}|R|
    \leq2^d\sum_{\substack{R\in\mc T\\R\subset P}}|Q_R|
    \leq {2^d} \eta^{-1}\sum_{\substack{R\in\mc T\\R\subset P}}|E_{Q_R}|
    \leq {2^d} \eta^{-1}|3P|
    = {6^d} \eta^{-1}|P|.
\]
Thus $\mc T$ is $\frac{6^d}{\eta}$-Carleson and therefore $\frac{\eta}{6^d}$-sparse, see \cite{LN2019}.
\end{proof}

\subsection{Extensions of the auxiliary lemmas}

We next extend Lemmas \ref{lemma:multistoptime} and \ref{Lemma:KeyRandomization} to the asymmetric cube pairs $(3Q,Q)$. The  additional difficulty is that the cubes $3R$ corresponding to disjoint dyadic cubes $R$ need not be disjoint.

\begin{lemma}\label{Lemma:AsymmetricStoppingTime}
 Let $1\leq r,s<\infty$ and let $Q$ be a cube. Let $f_j\in \avgL^r(3Q)$ and $g_j\in \avgL^s(Q)$ for $j=1,\dots, n$. Then there exists a disjoint collection $\mathcal{F}_Q\subset \mathcal{D}(Q)$ such that
\[
    \sum_{R\in\mc F_Q}|R|\leq\tfrac{1}{2}|Q|,
    \qquad
    \sum_{R\in\mc F_Q}\ind_{3R}\leq2\cdot3^d.
\]
Moreover, if $P\in\mc D(Q)$ is not contained in any $R\in\mc F_Q$, then for $j=1,\ldots,n$
\[
    \|f_j\|_{\avgL^r(3P)}\lesssim_d n^{\frac1r}\|f_j\|_{\avgL^r(3Q)},\qquad
    \|g_j\|_{\avgL^s(P)}\lesssim_d n^{\frac1s}\|g_j\|_{\avgL^s(Q)}.
\]
\end{lemma}

\begin{proof}
Let $M_p$ denote the rescaled Hardy--Littlewood maximal operator. Let $c_d$ be a dimensional constant such that 
\[
    E_Q\coloneqq\bigcup_{j=1}^n\left\{M_r(\ind_{3Q}f_j)>(c_dn)^{\frac1r}\|f_j\|_{\avgL^r(3Q)}\right\}
    \cup\bigcup_{j=1}^n\left\{M_s(\ind_Qg_j)>(c_dn)^{\frac1s}\|g_j\|_{\avgL^s(Q)}\right\}
\]
satisfies $|E_Q|\leq\frac{1}{2}|Q|$ by the weak boundedness of the maximal operator. Let $\mc F_Q$ be the collection of maximal cubes $R\in\mc D(Q)$ such that $9R\subset E_Q$. Then the cubes in $\mc F_Q$ are disjoint and
\[
    \sum_{R\in\mc F_Q}|R|\leq|E_Q|\leq\tfrac{1}{2}|Q|.
\]
By maximality, the claimed estimates hold with implicit constants $(3^dc_d)^{\frac1r}$ and $(9^dc_d)^{\frac1s}$, respectively.

To prove the overlap estimate, fix $x\in E_Q$ and set $\mc F_Q(x)\coloneqq\{R\in\mc F_Q:x\in3R\}.$
If $R,R'\in\mc F_Q(x)$ and $\ell(R)\leq\frac{1}{4}\ell(R')$, then the dyadic parent $\widehat{R}$ of $R$ satisfies
$9\widehat{R}\subset9R'\subset E_Q,$
contradicting the maximality of $R$. Hence the cubes in $\mc F_Q(x)$ belong to at most two consecutive generations. At each generation, at most $3^d$ of the dilated cubes contain $x$, and therefore
\[
    \sum_{R\in\mc F_Q}\ind_{3R}(x)\leq2\cdot3^d.
\]
This finishes the proof.
\end{proof}

\begin{lemma}\label{Lemma:AsymmetricRandomization}
Let $p\in(1,\infty)$ and let $\Lambda\colon L^p(\R^d)\times L^{p'}(\R^d)\to\C$ be a continuous bilinear form. Fix a cube $Q$, let $\mathcal{F}_Q\subset \mathcal{D}(Q)$ be a collection of disjoint cubes such that
\[
    \sum_{R\in\mc F_Q}\ind_{3R}\leq 2\cdot3^d.
\]
Let $\{\varepsilon_R\}_{R\in \mathcal{F}_Q}$ be independent Bernoulli random variables with $\P(\cbrace{\varepsilon_R=0})= \P(\cbrace{\varepsilon_R=1})=\frac12$. Set $G_Q\coloneqq Q\setminus\bigcup_{R\in\mc F_Q}R$,
and
\[m_\varepsilon(x)\coloneqq\ind_{3Q}(x)
    \prod_{\substack{R\in\mc F_Q:x\in3R}}2{\varepsilon_R},\qquad
    n_\varepsilon\coloneqq\ind_{G_Q}+2\sum_{R\in\mc F_Q}(1-\varepsilon_R)\ind_R,
\]
where the empty product equals $1$. Then $0\leq m_\varepsilon \leq 2^{2\cdot 3^d}$, $0\leq n_\varepsilon\leq 2$ and for $f\in L^p(3Q)$ and $g\in L^{p'}(Q)$ we have
\[
    \E\left[\Lambda(m_\varepsilon f,n_\varepsilon g)\right]
    =\Lambda(f,g)-\sum_{R\in\mc F_Q}\Lambda(\ind_{3R}f,\ind_Rg).
\]
Moreover, for every $R\in\mc F_Q$, either $m_\varepsilon$ vanishes on $3R$ or $n_\varepsilon$ vanishes on $R$.
\end{lemma}

\begin{proof}
By the overlap assumption, at most $2\cdot3^d$ factors in the definition of $m_\varepsilon(x)$ differ from $1$, proving the bound on $m_\varepsilon$.
Moreover, since the cubes in $\mc F_Q$ are disjoint, we have $0\leq n_\varepsilon\leq2$. Thus a similar continuity argument as in the proof of Lemma \ref{Lemma:KeyRandomization} applies.

For every $x\in\R^d$, independence gives
\[
    \E[m_\varepsilon(x)]=\ind_{3Q}(x).
\]
Furthermore, for every $R\in\mc F_Q$,
\[
    2\E\left[(1-\varepsilon_R)m_\varepsilon\right]=\ind_{3Q\setminus3R}.
\]
Indeed, the left-hand side vanishes on $3R$ since $\varepsilon_R(1-\varepsilon_R)=0$, while outside $3R$ the random variable $\varepsilon_R$ is independent of $m_\varepsilon$.
We therefore obtain, as in the proof of Lemma \ref{Lemma:KeyRandomization},
\begin{align*}
    \E\left[\Lambda(m_\varepsilon f,n_\varepsilon g)\right]
    &=\E\left[\Lambda(m_\varepsilon f,\ind_{G_Q}g)\right]
      +2\sum_{R\in\mc F_Q}\E\left[(1-\varepsilon_R)\Lambda(m_\varepsilon f,\ind_Rg)\right]\\
    &=\Lambda(f,\ind_{G_Q}g)+\sum_{R\in\mc F_Q}\Lambda(\ind_{3Q\setminus3R}f,\ind_Rg)\\
    &=\Lambda(f,g)-\sum_{R\in\mc F_Q}\Lambda(\ind_{3R}f,\ind_Rg).
\end{align*}
Finally, if $\varepsilon_R=0$, then $m_\varepsilon$ vanishes on $3R$, while if $\varepsilon_R=1$, then $n_\varepsilon$ vanishes on $R$. This proves the final assertion.
\end{proof}

\subsection{Proof of the main non-dyadic result}
Having established the auxiliary results for asymmetric cube pairs, we can now prove the main result for sparse domination over arbitrary cubes. The argument is a straightforward adaptation of the proof of Theorem \ref{Thm:SparseToConvexBody}.

\begin{theorem}\label{Thm:SparseToConvexBodynondyadic}
Let $1\leq r,s<\infty$ with $\frac{1}{r}+\frac{1}{s}>1$. Let $\Lambda\colon L^\infty_c(\R^d) \times L^\infty_c(\R^d) \to \C$ be a bilinear form. Suppose that there is a constant $C_{\mathrm{Sp}}$  such that for all $f,g \in L^\infty_c(\R^d)$ 
 there is an $\eta$-sparse family of cubes $\mathcal{S}$ for which
    \[
        |\Lambda(f,g)|\leq C_{\mathrm{Sp}}\sum_{Q\in\mathcal{S}}|Q|\,\|f\|_{\avgL^r(Q)}\|g\|_{\avgL^s(Q)}.
    \]
    Then for all $\vec f,\vec g \in L^\infty_c(\R^d,\C^n)$, there exists a $\frac{1}{2}$-sparse family $\mathcal{S}\subset \mc{D}$ for which
    \[
        |\Lambda(\vec f,\vec g)|\lesssim_{d,r,s} n^{1+\frac{1}{r}+\frac{1}{s}}\eta^{-1}C_{\mathrm{Sp}}\sum_{Q\in\mathcal{S}}|Q|\llangle \vec f\rrangle_{\avgL^r(3Q)}\cdot\llangle \vec g\rrangle_{\avgL^s(Q)}.
    \] 
\end{theorem}

Although the sparse family in the hypothesis of Theorem \ref{Thm:SparseToConvexBodynondyadic} consists of arbitrary cubes, the family in the conclusion may be chosen from any dyadic lattice $\mc{D}$. Of course, in order to remove the dilated cube $3P$ from the conclusion, one may also use the family $\mc{T}:=\cbrace{3P:P \in \mc{S}}$, which is $\frac{1}{2\cdot 3^d}$-sparse.

\begin{proof}[Proof of Theorem \ref{Thm:SparseToConvexBodynondyadic}]
We start by strengthening the initial sparse hypothesis. For $f,g \in  L^\infty_c(\R^d)$, let $\mc S_0$ be the $\eta$-sparse family from the hypothesis. Defining the bisublinear maximal operator as
\[
    \mc M_{r,s}(f,g)(x)\coloneqq\sup_{Q\ni x}\|f\|_{\avgL^r(Q)}\|g\|_{\avgL^s(Q)}, \qquad x \in \R^d,
\]
we have, using sparseness and \cite[(2.1)]{LN2022}, that there exists a $\frac{1}{2}$-sparse family $\mathcal{S}_1$ such that
\[
    \sum_{Q\in\mc S_0}|Q|\,\|f\|_{\avgL^r(Q)}\|g\|_{\avgL^s(Q)}
    \leq \eta^{-1}\int_{\R^d}\mc M_{r,s}(f,g)
    \lesssim_{d,r,s}\eta^{-1} \sum_{Q\in\mc S_1}|Q|\,\|f\|_{\avgL^r(Q)}\|g\|_{\avgL^s(Q)}.
\]
Applying Lemma \ref{Lemma:OneDyadicLattice} to $\mc S_1$ gives a $\frac{1}{2\cdot6^d}$-sparse family $\mc S_2\subset\mc D$ such that
\begin{equation}\label{Eq:GlobalAsymmetricSparse}
    |\Lambda(f,g)|
    \lesssim_{d,r,s}\eta^{-1}C_{\mathrm{Sp}}
    \sum_{Q\in\mc S_2}|Q|\,\|f\|_{\avgL^r(3Q)}\|g\|_{\avgL^s(Q)}.
\end{equation}
The proof of Proposition \ref{Prop:GlobalToUniformLocal} applies verbatim to \eqref{Eq:GlobalAsymmetricSparse}, with $3P$ and $3Q$ in the $f$-averages. Thus, for every $Q\in\mc D$,  $f\in L^\infty(3Q)$ and $g\in L^\infty(Q)$, there is an  $\frac{1}{4\cdot6^d}$-sparse family $\mc S_Q\subset\mc D(Q)$, such that
\begin{equation}\label{Eq:UniformLocalAsymmetricSparse}
    |\Lambda(f,g)|
    \lesssim_{d,r,s}\eta^{-1}C_{\mathrm{Sp}}
    \sum_{P\in\mc S_Q}|P|\,\|f\|_{\avgL^r(3P)}\|g\|_{\avgL^s(P)}.
\end{equation}
As in  Section \ref{sec:dyadic}, the sparse bound also shows that $\Lambda$ extends uniquely to a continuous bilinear form $L^p(\R^d)\times L^{p'}(\R^d) \to \C$ for all $p \in (r,s')$. 

The remainder of the proof is a verbatim adaptation of the proof of Theorem \ref{Thm:UniLocSparseToConvexBody}, using Lemmas \ref{Lemma:AsymmetricStoppingTime} and \ref{Lemma:AsymmetricRandomization} in place of Lemmas \ref{lemma:multistoptime} and \ref{Lemma:KeyRandomization}, respectively. We replace every localization and average of $f$ over a cube $Q$ by the corresponding localization and average over $3Q$, and replace the bound $0\leq m_\varepsilon(\omega)\leq2$ by
$ 0\leq m_\varepsilon(\omega)\leq2^{2\cdot3^d}.$
The bound for $n_\varepsilon$ remains unchanged, and the fixed sparseness parameter $\frac{1}{4\cdot6^d}$ is absorbed into the dimensional implicit constant. 
\end{proof}
\begin{remark}
    Suppose that $\Lambda\colon L^\infty_c(\R^d) \times L^\infty_c(\R^d) \to \C$ is a sesquilinear form satisfying the sparse domination assumption of Theorem \ref{Thm:SparseToConvexBodynondyadic}. Then the bilinear form $\tilde\Lambda(f,g)=\Lambda(f,\overline{g})$ also satisfies the same assumption.
Therefore, applying Theorem \ref{Thm:SparseToConvexBodynondyadic} to the bilinear form yields
\[
    |\Lambda(\vec f,\vec g)|=|\tilde\Lambda(\vec f,\overline{\vec g}\,)|\lesssim_{d,r,s} n^{1+\frac{1}{r}+\frac{1}{s}}\eta^{-1}C_{\mathrm{Sp}}\sum_{Q\in\mathcal{S}}|Q|\llangle \vec f\rrangle_{\avgL^r(3Q)}\cdot\llangle \overline{\vec g}\rrangle_{\avgL^s(Q)}.
\]
A simple calculation shows that $\llangle \overline{\vec g}\rrangle_{\avgL^s(Q)}=\{\overline{a}\,\colon\,a\in\llangle \vec g\rrangle_{\avgL^s(Q)}\}
$, and hence the above estimate is the standard convex body domination for sesquilinear forms.
\end{remark}

\subsection*{Acknowledgements}
This project was inspired by a presentation of Sandra Pott \cite{Pott2025}, in which she explained the beautiful connection between Bloom-type and matrix-weighted estimates. This led us to conjecture that commutator sparse domination may imply convex body domination, for which we then established a proof. Afterwards, we were able to weaken the assumption from commutator sparse domination to standard sparse domination.

\subsection*{AI disclosure statement}
 GPT-5.6 Sol Pro, accessed through ChatGPT by OpenAI, was used to prototype proof strategies for this note. More precisely, in a long initial conversation with ChatGPT 5.6 we iteratively constructed a proof that sparse domination for commutators $[b,T]$ as in Corollary \ref{Cor:SparseToCommutatorSparseintro} implies convex body domination for $T$. In this proof the function $b$ played the role of the functions $m_\varepsilon$ and $n_\varepsilon$ of Lemma \ref{Lemma:AsymmetricRandomization}. After we suggested  replacing the role of the function $b$ in that proof by Rademacher random variables, a rough version of the proof of Theorem \ref{Thm:UniLocSparseToConvexBody} was found in collaboration with ChatGPT 5.6.  The proof in its final form, and its extension to the non-dyadic setting, was fully developed, verified, and written by the authors and subsequently checked for typos and errors using ChatGPT 5.6. The authors take full responsibility for all claims and arguments in the paper.

\bibliographystyle{abbrv}
\bibliography{literature}
\end{document}